\documentclass[11pt,twoside,reqno]{amsart}

\usepackage{a4wide}
\usepackage[T1]{fontenc}
\usepackage[utf8]{inputenc}
\usepackage{amsmath}
\usepackage{amsfonts}
\usepackage{amssymb}
\usepackage{mathtools}
\usepackage{amsthm}
\usepackage{xcolor}
\usepackage{graphicx}
\usepackage[hidelinks]{hyperref}
\usepackage{standalone}
\usepackage{esint}
\usepackage{comment}
\usepackage[left=2.5cm,top=3cm,right=2.5cm]{geometry} 
\usepackage{bm}
\usepackage{enumerate}
\usepackage{enumitem}
\usepackage{cite}

\newcommand{\R}{\mathbb{R}}

\newcommand{\N}{\mathbb{N}}

\newcommand{\Z}{\mathbb{Z}}

\renewcommand{\epsilon}{\varepsilon}
\renewcommand{\rho}{\varrho}
\renewcommand{\phi}{\varphi}

\newcommand{\supp}{\mathrm{supp}}

\DeclareMathOperator{\area}{area}

\newcommand{\norm}[1]{\left\lVert#1\right\rVert}

\numberwithin{equation}{section}

\theoremstyle{plain}
\newtheorem{thm}{Theorem}[section]
\newtheorem{theorem}[thm]{Theorem}

\newtheorem{lemma}[thm]{Lemma}

\theoremstyle{definition}

\newtheorem{definition}[thm]{Definition}

\theoremstyle{remark}
\newtheorem{remark}[thm]{Remark}

\title{Quantum mixing for eigenfunctions of rational polygons in configuration space}

\author{Kai Hippi}
\address{Aalto University, Espoo, Finland}
\email{kai.hippi@aalto.fi}

\author{S{\o}ren Mikkelsen}
\address{University of Helsinki, Helsinki, Finland}
\email{soren.mikkelsen@helsinki.fi}

\begin{document}

\begin{abstract}
The Shnirelman–Zelditch–Colin de Verdi\`ere theorem and its weak mixing extension relate quantum ergodicity and quantum mixing with ergodicity and weak mixing of the geodesic flow. Integrable systems, such as the flat torus, do not satisfy either in general. Restricting to position-dependent observables, Marklof and Rudnick established equidistribution for almost all eigenfunctions of rational polygons. In this note, we extend this result to off-diagonal elements for a subset of rational polygons not including the torus using weak mixing of the directional billiard flow for almost all directions. Additionally, we provide a different proof that establishes the same result for $2$-tori.
\end{abstract}

\maketitle

\section{Introduction}

Quantum chaos studies the quantum systems arising from the classical systems via quantisation; especially, it tries to understand the implications of the classical chaos on the quantum side. One of the main challenges of quantum chaos is to understand the eigendata of the Hamiltonian, the energy operator of the system; the eigenvalues and eigenfunctions of the Hamiltonian yield an easy way to understand the evolution of the quantum system governed by the Schrödinger equation. 

The ergodic theorem of Shnirelman \cite{shnirelmanErgodicPropertiesEigenfunctions}, Zelditch \cite{zelditchUniformDistributionEigenfunctions1987}, and Colin de Verdi\`ere \cite{colindeverdiereErgodiciteFonctionsPropres1985} for compact Riemannian manifolds is a cornerstone result of quantum chaos. It states that the quantum system with Hamiltonian being the Laplace-Beltrami operator is quantum ergodic if the geodesic flow on the classical side is ergodic. Zelditch \cite{Zel96a} extended this to show that if the geodesic flow is weakly mixing on the classical side, this implies the quantum weak mixing property on the quantum side. A generalisation of the previous theorem, stated in Theorems 1 and 2 of \cite{Zel05}, can be obtained from the works of many authors: \cite{Shn74, colindeverdiereErgodiciteFonctionsPropres1985, zelditchUniformDistributionEigenfunctions1987, Zel90, GL93, LS93, Zel96b, Zel96a, ZZ96, Sun97}. One of the main consequences of these generalised theorems is that they give an equivalence between the ergodicity and mixing properties on the classical and quantum side. This in particular implies that integrable systems such as 2-spheres and 2-tori are neither quantum ergodic nor quantum mixing. Recently, there have also been a number of results concerning quantum mixing in the large scale limit; see \cite{Hippi,bordenave2026quantummixinglargeschreier,hippi2026quantummixingschrodingereigenfunctions}.  

These previously mentioned results on the quantum side have all been for general zeroth order pseudodifferential operators. One could limit this class of observables by, for example, restricting to multiplication operators by smooth functions. For this restricted class of observables, the question would be the same. One example of systems in which such questions have been studied is rational polygons. These are not quantum ergodic; however, it was established by Marklof and Rudnick \cite{MR2879311} that none the less almost all eigenfunctions still equidistribute in configuration space. Our aim in this note is to extend this result and establish the other properties of Theorems 1 and 2 in \cite{Zel05} when restricting to configuration space.         

Our first results concern the properties connected to ergodicity. As is the case for the results in \cite{MR2879311}, these will be true for any rational polygon. The statement is as follows.

\begin{theorem}\label{thm:main_1}
    Assume that $D$ is a rational polygon. Let $\{\psi_n\}_{n\in\N}$ be an orthonormal basis for the Dirichlet Laplacian on $D$ with the corresponding sequence of eigenvalues $\{\lambda_n^2\}_{n\in\N}$. Then for any $a\in C^\infty_c(D)$ the following two properties are true.  
    \begin{enumerate}[label=$(\arabic*)$]
        \item
        \begin{equation*}
        \lim_{\lambda\rightarrow \infty}\frac{1}{\mathcal{N}(\sqrt{-\Delta_D},\lambda)} \sum_{\lambda_n \leq \lambda} \Big| \langle a \psi_n , \psi_n\rangle - \frac{1}{\area(D)} \int_{D} a(x)  \, dx \Big|^2 =0.
        \end{equation*}
        \item For every $\varepsilon>0$ there exists $\delta(\varepsilon)>0$ such that
        \begin{equation*}
        \limsup_{\lambda\rightarrow \infty}\frac{1}{\mathcal{N}(\sqrt{-\Delta_D},\lambda)} \sum_{\substack{j\neq k: \lambda_j,\lambda_k\leq \lambda  \\ |\lambda_{j}-\lambda_{k}| < \delta(\varepsilon)}} \big| \langle a \psi_j , \psi_k\rangle  \big|^2 <\varepsilon.
        \end{equation*}
    \end{enumerate}
    In both cases $\mathcal{N}(\sqrt{-\Delta_D},\lambda)$ is the number of eigenvalues of $\sqrt{-\Delta_D}$ less than or equal to $\lambda$. 
\end{theorem}

Our next result will concern the properties related to weak mixing. These results will not be true for all rational polygons, but will exclude a subset of rational polygons. This subset of polygons is defined in the following definition. 

\begin{definition}
    Let $\mathcal{E}$ be the set of all almost integrable polygons and all polygons with all angles in $\{\frac{\pi}{2},\frac{3\pi}{2}\}$ and if orientated to have horizontal/vertical sides, all horizontal lengths or all vertical lengths are commensurable.
\end{definition}

The integrable polygons are all rectangles, and the $(\frac{\pi}{2},\frac{\pi}{4},\frac{\pi}{4})$, $(\frac{\pi}{3},\frac{\pi}{3},\frac{\pi}{3})$ and $(\frac{\pi}{2},\frac{\pi}{3},\frac{\pi}{6})$ triangles. An almost integrable polygon is a polygon drawn on the grid, or lattice in the terminology of Gutkin, spanned by reflecting one of the completely integrable polygons along its side ad infinitum. A set of numbers is commensurable if all ratios are rational. With this, we can state our next result.

\begin{theorem}\label{thm:main_2}
    Assume that $D$ is a rational polygon that does not belong to $\mathcal{E}$. Let $\{\psi_n\}_{n\in\N}$ be an orthonormal basis for the Dirichlet Laplacian on $D$ with the corresponding sequence of eigenvalues $\{\lambda_n^2\}_{n\in\N}$. Then for any $a\in C^\infty_c(D)$ the following property is true.  
    \begin{itemize}
        \item[$(3)$]
        For every $\varepsilon>0$ there exists $\delta(\varepsilon)>0$ such that for all $\tau\in\R$
        \begin{equation*}
        \limsup_{\lambda\rightarrow \infty}\frac{1}{\mathcal{N}(\sqrt{-\Delta_D},\lambda)} \sum_{\substack{j\neq k: \lambda_j,\lambda_k\leq \lambda  \\ |\lambda_{j}-\lambda_{k}-\tau| < \delta(\varepsilon)}} \big| \langle a \psi_j , \psi_k\rangle  \big|^2 <\varepsilon,
        \end{equation*}
    \end{itemize}
    where $\mathcal{N}(\sqrt{-\Delta_D},\lambda)$ is the number of eigenvalues of $\sqrt{-\Delta_D}$ less than or equal to $\lambda$. 
\end{theorem}

\begin{remark}
$ $ \\ 
\vspace{-1em}
\begin{enumerate}[label=\textnormal{(\roman*)}]
    \item  Theorems~\ref{thm:main_1} and \ref{thm:main_2} also hold for the Neumann Laplacian on $D$ and more generally for an arbitrary translation surface with the obvious restriction in the case of Theorem~\ref{thm:main_2}.

    \item The proofs of both Theorems~\ref{thm:main_1} and \ref{thm:main_2} can be generalised to pseudodifferantial operators of order zero under the assumption that the billiard flow of the polygon is ergodic or weakly mixing in phase space. Note that there exist polygons with weakly mixing billiard flow; see \cite{MR5067163}.
\end{enumerate}
\end{remark}

When considering Property (3) of the previous theorem and the proof, one would expect that this property will not be true for tori, since for this case the directional billiard flow is not weakly mixing, as established in \cite{aranaherrera2024weakmixingrationalbilliards} and considering the equivalences established in \cite{Zel05}. However, as is evident from the next theorem, this is not the case. Before we state this result, we will introduce some notation. For $l\in \R^2_{>0}$ we denote by $\Gamma_l$ the lattice
\begin{equation*}
    \Gamma_l = ( l_1 \Z) \times ( l_2 \Z)
\end{equation*}
and the dual lattice by $\Gamma_l^{*}$. With this notation, we have the following theorem.

\begin{theorem}
    \label{Thm:Main_theorem_tori}
    Let $l\in \R^2_{>0}$ and set $\mathbb{T}_{l} = \R^2 \big/ \Gamma_{l}$. Moreover, let $\{\psi_{n}\}_{n\in\Gamma_l^{*}}$ be an orthonormal basis for the Laplace-Beltrami operator $-\Delta_{\mathbb{T}_{l}}$ with the corresponding sequence of eigenvalues $\{\lambda_{n}^2\}_{n\in\Gamma_l^{*}}$. 
    Then for any $a\in L^2(\mathbb{T}_{l})$ points $(1)-(3)$ of Theorem~\ref{thm:main_1} and Theorem~\ref{thm:main_2} are valid.
\end{theorem}

\begin{remark}
$ $ \\ 
\vspace{-1em}
    \begin{enumerate}[label=\textnormal{(\roman*)}]
    \item Since the proof does not use microlocal analysis, we are able to consider these more general $L^2$ observables. Hence, we will also reprove Properties (1) and (2) since Theorem~\ref{thm:main_1} only establishes this for smooth observables.
    
    \item The result in Theorem~\ref{Thm:Main_theorem_tori} generalises to tori of any dimension greater than or equal to $2$.

    \item We remark that this theorem contrasts the setting of large scale tori, where Properties (2) and (3) are not true as demonstrated in the appendix of \cite{Hippi}.  

    \item It could be of interest to find a rational polygon such that Property (3) of Theorem~\ref{thm:main_2} is not valid.

    \end{enumerate}
\end{remark}

\subsection*{Acknowledgements} The authors thank Jens Marklof and Ze\'ev Rudnik for useful discussions and comments during the preparation of this manuscript. The authors acknowledge support from the Research Council of Finland's Academy Research Fellowship \emph{``Quantum chaos of large and many body systems''}, grant Nos. 347365, 353738. K.H. is also supported by the Vilho, Yrjö, and Kalle Väisälä Foundation. The authors would like to thank the Isaac Newton Institute for Mathematical Sciences, Cambridge, for support and hospitality during the programme Geometric Spectral Theory and Applications, where work on this manuscript was undertaken. This work was supported by EPSRC grant EP/Z000580/1.

\section{Preliminaries}
In this section, we collect some preliminary results needed to establish the main results. However, before we state the specific preliminaries needed for the questions studied here, we will state the following general lemma, which may be of independent interest.

\begin{lemma}\label{LE:exstension_to_other_basis}

    Let $(M,g)$ be a compact Riemannian manifold with Laplace-Beltrami operator $-\Delta_g$. Let $\{ \psi_n \}$ and $\{ \phi_n \}$ be two orthonormal eigenbases of $-\Delta_g$. Then:
    \begin{enumerate}
        \item [$(a)$] if Properties $(1)$ and $(2)$ hold for $\{ \psi_n\}$, they also hold for $\{ \phi_n \}$,
        \item [$(b)$] if Properties $(1)$, $(2)$, and $(3)$ holds for $\{ \psi_n\}$, they also hold for $\{ \phi_n \}$.
    \end{enumerate}
    
\end{lemma}
This lemma is generalisable to pseudodifferential operators of order $0$ with a similar argument.
\begin{proof}

    Begin by proving Property (1) in (\emph{a}). We have
    \begin{align*}
        &
        \sum_{ j : \lambda_j \leq \lambda } | \langle (a - \bar{a}) \phi_j, \phi_j \rangle |^2 \leq \sum_{ \sqrt{\mu} \leq \lambda } \| \Pi_\mu (a - \bar{a}) \Pi_\mu \|_{HS}^2 = \sum_{ \sqrt{\mu} \leq \lambda } \, \sum_{ j,k: \lambda_j^2 = \lambda_k^2 = \mu } | \langle (a - \bar{a}) \psi_j, \psi_k \rangle |^2,
    \end{align*}
    where $\Pi_\mu$ is an orthonormal projection onto $V_\mu = \operatorname{span} \{  \psi_n : \lambda_n^2 = \mu \}$, and where $\overline{a}$ is the integral of $a$ with respect to the normalised volume form. The last inequality followed since the Hilbert-Schmidt norm is basis independent. This is bounded from above for any $\delta > 0$ by
    \begin{align}
        &
        \label{Expr: Lemma 2.1}
        \sum_{\sqrt{\mu} \leq \lambda } \, \sum_{j : \lambda_j^2 = \mu} | \langle (a - \bar{a}) \psi_j, \psi_j \rangle |^2 \ + \  \sum_{ \sqrt{\mu}, \sqrt{\mu'} \leq \lambda } \, \sum_{ \substack{  j \neq k: \lambda_j^2 = \mu, \lambda_k^2 = \mu', \\ | \sqrt{\mu} - \sqrt{\mu'}| < \delta} } | \langle (a - \bar{a}) \psi_j, \psi_k \rangle |^2.
    \end{align}
    Now Property (1) can be proven for $\{ \phi_n \}$ using Properties (1) and (2) for $\{ \psi_n\}$.

    Next, we prove Property (2) in (\emph{a}). We have
    \begin{align*}
        &
        \sum_{ \substack{ j \neq k : \lambda_j , \lambda_k \leq \lambda, \\ | \lambda_j - \lambda_k | < \delta} } | \langle a \phi_j, \phi_k  \rangle |^2 = \sum_{ \substack{ j \neq k : \lambda_j , \lambda_k \leq \lambda, \\ | \lambda_j - \lambda_k | < \delta} } | \langle (a - \bar{a}) \phi_j, \phi_k  \rangle |^2 \leq  \sum_{ \substack{\sqrt{\mu}, \sqrt{\mu'} \leq \lambda, \\ | \sqrt{\mu}  - \sqrt{\mu'} | < \delta } }  \| \Pi_\mu (a - \bar{a}) \Pi_{\mu'} \|_{HS}^2.
    \end{align*}
    Using that the Hilbert-Schmidt norm is basis independent, the previous equals \eqref{Expr: Lemma 2.1} for a fixed $\delta$. To deal with the first term, we can use the previous part of the proof. The latter term is dealt with by using Property (2) for $\{  \psi_n\}$.

    Finally, we prove (\emph{b}). By (\emph{a}) what remains is Property $(3)$ for $\tau\neq 0$. Hence, let $\tau\neq 0$ be given. We have for small enough $\delta$
    \begin{align*}
        &
        \sum_{ \substack{ j \neq k : \lambda_j , \lambda_k \leq \lambda, \\ | \lambda_j - \lambda_k - \tau| < \delta} } | \langle a \phi_j, \phi_k  \rangle |^2 = \sum_{\substack{ \sqrt{\mu}, \sqrt{\mu'} \leq \lambda, \\ | \sqrt{\mu} - \sqrt{\mu'} - \tau | < \delta }} \| \Pi_\mu a  \Pi_{\mu'} \|_{HS}^2 = \sum_{ \substack{ j \neq k : \lambda_j , \lambda_k \leq \lambda \\ | \lambda_j - \lambda_k - \tau| < \delta} } | \langle a \psi_j, \psi_k  \rangle |^2
        ;
    \end{align*}
    having $\delta$ small enough means that we will not need to consider diagonal elements. Using Property (3) for $\{ \psi_n\}$ gives the wanted conclusion.   
\end{proof}

\subsection{Results on billiards in rational polygons}
The phase space for billiards in a rational polygon $D$ is the unit cotangent bundle $S^{*}D$, which is the direct product 
\begin{equation*}
    S^{*}D = D\times S^1.
\end{equation*}
The normalised Liouville measure is identified as
\begin{equation*}
    d\mu(x,\omega) = \frac{1}{\area(D)} dx d\omega,
\end{equation*}
where $dx$ is the Lebesgue measure on $D$ and $d\omega$ is the Haar measure on $S^1$. The measure $d\mu$ is invariant under the billiard flow $\Phi^t$. This flow is defined in $S^{*}D$ through specular reflections for all trajectories not hitting the vertices of the polygon $D$. The reflection law for the trajectories hitting the vertices can be defined arbitrarily; however, these will form a set of measure zero and hence will be ignored in the following discussion. 

Let $G$ be the finite group generated by the linear parts of the reflections in the sides of the polygon $D$. That this group is finite follows from the assumptions that $D$ is simple and all vertex angles are rational multiplies of $\pi$. For each direction $\omega$, the set
\begin{equation*}
    D_\omega \coloneqq D\times \bigcup_{g\in G} \{g\omega\} 
\end{equation*}
is preserved by the flow. The directional flow $\Phi^t_{\omega}$ is the restriction of the flow in the direction $\omega$. Kerckhoff, Masur, and Smillie \cite{MR855297} established that for almost all directions the directional flow $\Phi^t_{\omega}$ is uniquely ergodic. More recently, Arana-Herrera, Chaika, and Forni \cite{aranaherrera2024weakmixingrationalbilliards} have further established that if $D$ does not belong to the set $\mathcal{E}$, then for almost all directions the directional flow is weakly mixing. Using this, we immediately obtain the following lemma.
\begin{lemma}\label{LE:consequence_weak_mixing}
    Let $D$ be a rational polygon that does not belong to $\mathcal{E}$ and let $\tau\in\R\setminus\{0\}$. Then for any $a\in C_0^\infty(D)$ we have
    \begin{equation*}
        \lim_{T\rightarrow \infty} \int_{S^*D} \bigg| \frac{1}{2T} \int_{-T}^{T} e^{it\tau} (  a\circ\Phi^t - \overline{a}) \, dt \bigg|^2 \, d\mu = 0,
    \end{equation*}
    where $\overline{a}$ is the integral of $a$ with respect to the normalised Lebesgue measure on $D$.
\end{lemma}
\begin{proof}
For notational convenience, we will suppose that $a$ is real valued, to avoid some complex conjugates. Since $G$ is a finite group and the Haar measure in $S^1$ is invariant under the action of elements in $G$, we have the following
\begin{equation*}
    \begin{aligned}
     \MoveEqLeft \int_{S^*D} \bigg| \frac{1}{2T} \int_{-T}^{T} e^{it\tau} (  a\circ\Phi^t(x,\omega) - \overline{a}) \, dt \bigg|^2 \, d\mu 
     \\
     ={}& \frac{1}{\# G} \sum_{g\in G}\int_{S^*D} \bigg| \frac{1}{2T} \int_{-T}^{T} e^{it\tau} (  a\circ\Phi^t(x,g\omega) - \overline{a}) \, dt \bigg|^2 \, d\mu
     \\
     ={}& \Bigg| \frac{1}{\# G} \sum_{g\in G}\int_{S^*D}  \frac{1}{4T^2} \int_{-T}^{T}\int_{-T}^{T} e^{i(t-s)\tau} (  a\circ\Phi^t(x,g\omega) - \overline{a})(  a\circ\Phi^s(x,g\omega) - \overline{a}) \, dt ds d\mu \Bigg|
     \\
     \leq{}& \int_{S^1} \frac{1}{4T^2} \int_{-T}^{T}\int_{-T}^{T} \Bigg| \frac{1}{\# G \area(D)}  \sum_{g\in G} \int_D a\circ\Phi^t(x,g\omega) a\circ\Phi^s(x,g\omega) \,dx  - \overline{a}^2 \Bigg| \, dtds d\omega
     \\
     &+2 \int_{S^1} \frac{1}{2T} \int_{-T}^{T}  \Bigg| \frac{1}{\# G \area(D)} \sum_{g\in G}\int_{D}  a\circ\Phi^t(x,g\omega) \overline{a} \, dx  - \overline{a}^2 \Bigg| \, dt d\omega 
     \end{aligned}
\end{equation*}
Since $a$ is independent of $\omega$, it follows from weak mixing for almost all directions that the two integrals on the left-hand side go to zero as $T$ goes to infinity. 
\end{proof}

The next lemma is a version of Lemma~{2} from \cite{ZZ96}, where a proof can also be found. 

\begin{lemma}\label{LE:Set_for_flow}
    Let $D$ be a rational polygon. Then for any $T>0$ there exists a subset $X_T \subset S^{*}D$, such that $\Phi^t$ is well defined for all $|t|\leq T$. Moreover, the set $X_T$ is open and has full measure.
\end{lemma}

\subsection{Some results from microlocal analysis}
In this subsection, we recall some notation and results from microlocal analysis that will be needed later. For an introduction to microlocal analysis, see, e.g. \cite{MR0618463,MR2304165}. Firstly, by $\Psi_{\mathrm{phg}}$, we will denote the set of zeroth order polyhomogeneous pseudo-differential operators. For an operator $A\in \Psi_{\mathrm{phg}}$ we denote the principal symbol by $\sigma_0(A)$ and we denote by $WF(A)$ the invariantly defined essential support of the full symbol. 
The following three lemmas are Lemma {3}, {4}, and {5} in \cite{ZZ96}, where proofs can also be found.

\begin{lemma}\label{LE:weak_convergence}
    Let $D$ be a rational polygon, and let $-\Delta_D$ be the Dirichlet Laplacian on $D$. For any orthonormal basis of eigenfunctions of $-\Delta_D$,
    \begin{equation*}
        \frac{1}{\mathcal{N}(\sqrt{-\Delta_D},\lambda)} \sum_{\lambda_j\leq \lambda} |\psi_j|^2 \underset{\lambda \rightarrow \infty}{\longrightarrow} 1 \qquad \text{{\textit weakly in }} L^1(D).
    \end{equation*}
\end{lemma}

\begin{lemma}\label{LE:Weyl_law_gene}
     Let $D$ be a rational polygon, and let $A\in \Psi_{\mathrm{phg}}(D)$ with Schwartz kernel compactly supported in $D^{\circ}\times D^{\circ}$. Then with the notation of Lemma~\ref{LE:weak_convergence} we have
     \begin{equation*}
          \lim_{\lambda\rightarrow \infty}\frac{1}{\mathcal{N}(\sqrt{-\Delta_D},\lambda)} \sum_{\lambda_j\leq \lambda} \langle A\psi_j,\psi_j\rangle = \int_{S^{*}D} \sigma_0(A) \, d\mu.
     \end{equation*}
\end{lemma}

\begin{lemma}\label{LE:Egorov}
    Let $D$ be a rational polygon and let $A\in \Psi_{\mathrm{phg}}(D)$ with Schwartz kernel compactly supported in $D^{\circ}\times D^{\circ}$ and in addition $WF(A) \cap S^{*}D \Subset X_T$. If $\varphi \in C_0^\infty(D^{\circ})$, then for $t\leq |T|$,
    \begin{equation*}
        A_t^\varphi \coloneqq \varphi e^{it\sqrt{-\Delta_D}} A  e^{-it\sqrt{-\Delta_D}} \varphi \in \Psi_{\mathrm{phg}}(D^{\circ}),
    \end{equation*}
    and
    \begin{equation*}
        \sigma_0( A_t^\varphi) = (\pi^{*} \varphi)^2 \sigma_0(A)\circ \Phi^t,
    \end{equation*}
    where $\pi:S^*D \mapsto D$ is the natural projection.
\end{lemma}

\section{Proof of Theorem~\ref{thm:main_1} and Theorem~\ref{thm:main_2}}

In this section, we prove Theorem \ref{thm:main_1} and Theorem \ref{thm:main_2}. We note that Property (1) has already been proven by Marklof and Rudnick \cite{MR2879311}. Hence, it remains to prove Properties (2) and (3). The proof will follow the approach of Zelditch and Zworski \cite{ZZ96} with some small changes in the beginning and end of the argument. Hence, we will give the full proof with these small modifications.

\begin{proof}[Proof of Theorem \ref{thm:main_2}]
We will start by letting $\varepsilon>0$, $\tau\in\R \setminus\{0\}$ and $a\in C_0^\infty(D)$ be given. The case $\tau=0$ is covered by Theorem~\ref{thm:main_1} Property (2).

To prove the result, we fix $T>0$ to be chosen arbitrary large later. We let $\{U_\eta\}_{\eta\in(0,\eta_0]}$ be a nested family of open sets such that 
\begin{equation*}
    U_\eta \Subset X_{2T}\setminus S^{*}_{\partial D_{reg}} D,\qquad \bigcup_{\eta\in(0,\eta_0]} U_\eta = X_{2T}, \qquad\text{and}\qquad \overline{U}_{\eta'} \subset U_\eta \qquad\text{for $\eta<\eta'$},
\end{equation*} 
where $X_{2T}$ is the set from Lemma~\ref{LE:Set_for_flow}. Then there exist operators $E_\eta \in \Psi^0_{\mathrm{phg}}(D^{\circ})$ with Schwartz kernel compactly supported in $D^{\circ}\times D^{\circ}$ such that
\begin{equation*}
    WF(E_\eta) \cap S^{*}D^{\circ} \subset U_{2\eta}, \qquad WF(E_\eta - I) \cap S^{*}D^{\circ} \subset S^{*}D^{\circ} \setminus U_{\eta}, \qquad\text{and}\qquad 0\leq \sigma_0(E_\eta)\leq 1.
\end{equation*}
We then set $A_\eta = a E_\eta $, which will have principal symbol $a\sigma_0(E_\eta)$. With this operator, we have
\begin{equation}\label{EQ_main_prop_3_1}
    \sum_{\substack{j\neq k: \lambda_j,\lambda_k\leq \lambda,  \\ |\lambda_{j}-\lambda_{k}-\tau| < \delta}} \big| \langle a \psi_j , \psi_k\rangle  \big|^2 
    \leq 2\sum_{\substack{j\neq k: \lambda_j,\lambda_k\leq \lambda,  \\ |\lambda_{j}-\lambda_{k}-\tau| < \delta}} \big| \langle A_\eta \psi_j , \psi_k\rangle  \big|^2 + \big| \langle (a- A_\eta) \psi_j , \psi_k\rangle  \big|^2.
\end{equation}
For the sum over the second term on the right-hand side of \eqref{EQ_main_prop_3_1} we have Parseval's identity and Lemma~\ref{LE:Weyl_law_gene} so that
\begin{equation}\label{EQ_main_prop_3_2}
    \begin{aligned}
    \MoveEqLeft \frac{1}{\mathcal{N}(\sqrt{-\Delta_D},\lambda)}\sum_{\substack{j\neq k: \lambda_j,\lambda_k\leq \lambda,  \\ |\lambda_{j}-\lambda_{k}-\tau| < \delta}} \big| \langle (a- A_\eta) \psi_j , \psi_k\rangle \big|^2 
    \\
    &\leq  \frac{1}{\mathcal{N}(\sqrt{-\Delta_D},\lambda)} \sum_{ \lambda_j \leq \lambda } \langle (a- A_\eta)^{*}(a- A_\eta) \psi_j , \psi_j\rangle
    \underset{\lambda \rightarrow \infty}{\longrightarrow} \int_{S^{*}D} |a-a\sigma_0(E_\eta) |^2 \, d\mu.
    \end{aligned}
\end{equation}
From Lemma~\ref{LE:Set_for_flow} we have that $X_{2T}$ has full measure for all $T>0$. Since $\sigma_0(E_\eta)=1$ on $U_\eta$ it follows that $\lim_{\eta\rightarrow 0} a\sigma_0(E_\eta)=a$ almost everywhere. Since $a-a\sigma_0(E_\eta)$ is uniformly bounded, it follows from \eqref{EQ_main_prop_3_2} using dominated convergence that
\begin{equation}\label{EQ_main_prop_3_3}
    \limsup_{\lambda\rightarrow\infty}  \sum_{\substack{j\neq k: \lambda_j,\lambda_k\leq \lambda,  \\ |\lambda_{j}-\lambda_{k}-\tau| < \delta}} \big| \langle a \psi_j , \psi_k\rangle  \big|^2 
    \leq h_1(T,\eta) + 2 \limsup_{\lambda\rightarrow\infty}\sum_{\substack{j\neq k: \lambda_j,\lambda_k\leq \lambda,  \\ |\lambda_{j}-\lambda_{k}-\tau| < \delta}} \big| \langle A_\eta \psi_j , \psi_k\rangle  \big|^2, 
\end{equation}
where $h_1(T,\eta) \rightarrow 0$ as $\eta\rightarrow0$. 
For our fixed $T$ we set $\delta = \frac{\pi}{2T}$ and let $\bar{a}$ be $a$ integrated with respect to the normalised Lebesgue measure on $D$. With these choices, we notice that
\begin{equation}\label{EQ_main_prop_3_4}
\begin{aligned}
    \sum_{\substack{j\neq k: \lambda_j,\lambda_k\leq \lambda,  \\ |\lambda_{j}-\lambda_{k}-\tau| < \delta(\varepsilon)}} \big| \langle A_\eta \psi_j , \psi_k\rangle  \big|^2 
    & = \sum_{\substack{j\neq k: \lambda_j,\lambda_k\leq \lambda,  \\ |\lambda_{j}-\lambda_{k}-\tau| < \delta(\varepsilon)}} \big| \langle (A_\eta - \bar{a}) \psi_j , \psi_k\rangle  \big|^2 
    \\
    &\leq 4\sum_{\substack{j\neq k: \lambda_j,\lambda_k\leq \lambda,  \\ |\lambda_{j}-\lambda_{k}-\tau| < \delta(\varepsilon)}} \big| \langle (A_\eta - \bar{a}) \psi_j , \psi_k\rangle  \big|^2 \Big|\tfrac{\sin((\lambda_{j}-\lambda_{k}-\tau) T)}{(\lambda_{j}-\lambda_{k}-\tau) T}\Big|^2 
    \\
    &= 4 \sum_{\substack{j\neq k: \lambda_j,\lambda_k\leq \lambda,  \\ |\lambda_{j}-\lambda_{k}-\tau| < \delta(\varepsilon)}} \big| \langle A_\eta(T,\tau) \psi_j , \psi_k\rangle  \big|^2.
    \end{aligned}
\end{equation}
where
\begin{equation*}
    A_\eta(T,\tau) = \frac{1}{2T} \int_{-T}^T e^{it\tau} e^{-i t \sqrt{\Delta_D}} (A_\eta - \bar{a}) e^{i t \sqrt{\Delta_D}} \,dt.
\end{equation*}
By Parseval's identity, we obtain the upper bound  
\begin{equation}\label{EQ_main_prop_3_5}
         \sum_{\substack{j\neq k: \lambda_j,\lambda_k\leq \lambda,  \\ |\lambda_{j}-\lambda_{k}-\tau| < \delta}} \big| \langle A_\eta(T,\tau) \psi_j , \psi_k\rangle  \big|^2 
         \leq \sum_{\lambda_j \leq \lambda} \langle ( A_\eta(T,\tau) )^{*}A_\eta(T,\tau) \psi_j , \psi_j\rangle. 
\end{equation}
For $\nu>0$ we introduce the functions $\varphi_i^\nu \in C_0^\infty(D^{\circ}; [0,1])$, $i\in\{1,2,3\}$ such that
\begin{equation*}
    d(\supp(\varphi_i^\nu -1),\partial D)<\delta \text{ for all $i\in\{1,2,3\}$} \qquad\text{and}\qquad \varphi_i^\nu=1 \text{ on $\supp (\varphi_{i+1}^\nu)$ for $i\in\{1,2\}$}.
\end{equation*} 
With these functions and by writing $B_{\eta,T,\tau} =( A_\eta(T,\tau))^{*}A_\eta(T,\tau) $ we have for all $j$
\begin{equation*}
    \begin{aligned}
    \langle B_{\eta,T,\tau} \psi_j , \psi_j\rangle = {}&  
    \langle \varphi_1^\nu B_{\eta,T,\tau} \varphi_2^\nu \psi_j , \psi_j\rangle
    +\langle (1-\varphi_1^\nu) B_{\eta,T,\tau} \varphi_2^\nu \psi_j , \psi_j\rangle
    \\
    &+\langle \varphi_3^\nu B_{\eta,T,\tau}(1- \varphi_2^\nu) \psi_j , \psi_j\rangle
    +\langle (1-\varphi_3^\nu) B_{\eta,T,\tau}(1- \varphi_2^\nu) \psi_j , \psi_j\rangle.
    \end{aligned}
\end{equation*}
As observed in \cite{ZZ96} it follows from the theorem on propagation of singularities along bicharacteristics intersecting the boundary transversally (see \cite[Theorem 24.2.1]{MR2304165}) that the operators $(1-\varphi_1^\nu) B_{\eta,T,\tau} \varphi_2^\nu $ and $\varphi_3^\nu B_{\eta,T,\tau}(1- \varphi_2^\nu)$ have kernels in $C^\infty(D\times D)$. Hence for any $N\in\N$, multiplying by the identity operator $I=(I-\Delta_D)^{-N}(I-\Delta_D)^{N}$ and using that $\psi_j$ is an eigenvector of $-\Delta_D$ we obtain
\begin{equation}
    \big|\langle (1-\varphi_1^\nu) B_{\eta,T,\tau} \varphi_2^\nu \psi_j , \psi_j\rangle +\langle \varphi_3^\nu B_{\eta,T,\tau}(1- \varphi_2^\nu) \psi_j , \psi_j\rangle \big| \leq C_{T,\nu,\eta,N} \lambda_j^{-N}.
\end{equation}
Now, using this observation and applying Lemma~\ref{LE:Egorov} in combination with Lemma~\ref{LE:weak_convergence} and Lemma~\ref{LE:Weyl_law_gene}, we obtain
\begin{equation*}
    \begin{aligned}
     \MoveEqLeft 
     \limsup_{\lambda\rightarrow\infty}\frac{1 }{\mathcal{N}(\sqrt{-\Delta_D},\lambda)} \sum_{\lambda_j \leq \lambda} \langle B_{\eta,T,\tau} \psi_j , \psi_j\rangle
     \\
    &\leq \int_{S^*D} \bigg| \frac{1}{2T} \int_{-T}^{T} e^{it\tau} \pi^{*}\varphi^\nu_2 (\sigma_0(A_\eta)\circ\Phi^t - \bar{a}) \, dt \bigg|^2 \, d\mu + C_{T,\eta} \int_{D} (1-\varphi_3^\nu(x)) \, dx 
    \end{aligned}
\end{equation*}
Combining this upper bound with \eqref{EQ_main_prop_3_5} we obtain
\begin{equation}\label{EQ_main_prop_3_6}
    \begin{aligned}
        \MoveEqLeft \limsup_{\lambda\rightarrow\infty} \frac{1}{\mathcal{N}(\sqrt{\Delta_D},\lambda)} \sum_{\substack{j\neq k: \lambda_j,\lambda_k\leq \lambda  \\ |\lambda_{j}-\lambda_{k}-\tau| < \delta}} \big| \langle A_\eta(T,\tau) \psi_j , \psi_k\rangle  \big|^2 
        \\
        &\leq  
        \int_{S^*D} \bigg| \frac{1}{2T} \int_{-T}^{T}e^{it\tau} ( \sigma_0(A_\eta)\circ\Phi^t - \bar{a}) \, dt \bigg|^2 \, d\mu  + h_{2}(T,\eta,\nu),
        \end{aligned}
\end{equation}
where $\lim_{\nu\rightarrow 0} h_{2}(T,\eta,\nu)=0$. Again using $\lim_{\eta\rightarrow 0} a\sigma_0(E_\eta)=a$ almost everywhere and that $a\sigma_0(E_\eta)$ is uniformly bounded in $\eta$ it follows from dominated convergence that
\begin{equation}\label{EQ_main_prop_3_7}
        \int_{S^*D} \bigg| \frac{1}{2T} \int_{-T}^{T} e^{it\tau} (\sigma_0(A_\eta)\circ\Phi^t - \bar{a}) \, dt \bigg|^2 \, d\mu  \leq  h_3(T,\eta)+ \int_{S^*D} \bigg| \frac{1}{2T} \int_{-T}^{T} e^{it\tau}  ( a\circ\Phi^t - \bar{a}) \, dt \bigg|^2 \, d\mu,
\end{equation}
where $\lim_{\eta\rightarrow 0} h_3(T,\eta)=0$. Moreover, from Lemma~\ref{LE:consequence_weak_mixing} it follows that the integral on the right-hand side of \eqref{EQ_main_prop_3_7} goes to zero as $T$ goes to infinity. Using this observation and combining the estimates in \eqref{EQ_main_prop_3_3}, \eqref{EQ_main_prop_3_6} we obtain 
\begin{equation}\label{EQ_main_prop_3_8}
    \limsup_{\lambda\rightarrow\infty}  \sum_{\substack{j\neq k: \lambda_j,\lambda_k\leq \lambda  \\ |\lambda_{j}-\lambda_{k}-\tau| < \delta}} \big| \langle a \psi_j , \psi_k\rangle  \big|^2 
    \leq h_1(T,\eta) + 4( h_{2}(T,\eta,\nu) +  h_3(T,\eta) +  h_4(T)).
\end{equation}
From this we see that by first fixing $T$ (and thereby also $\delta$) then $\eta$ and finally $\nu$ we can obtain that the right-hand side of \eqref{EQ_main_prop_3_8} is less than $\varepsilon$. This concludes the proof.
\end{proof}

\begin{proof}[Proof of Theorem~\ref{thm:main_1}]
    As mentioned the proof of Property (1) is already given in \cite{MR2879311}. For Property (2) the proof is almost identical to the one just given. The only difference is at the end of the proof when Lemma~\ref{LE:consequence_weak_mixing} is applied one should instead use Lemma~{2} from \cite{MR2879311} instead. This concludes the proof.
\end{proof}

\section{Proof of Theorem~\ref{Thm:Main_theorem_tori}}

Before we give a proof of Theorem~\ref{Thm:Main_theorem_tori} we will state and prove the following counting type lemma, which will be used in the proof of the theorem.

\begin{lemma}\label{LE:simple_counting_lemma}
    Let $l\in\R_{>}^2$ and let $k \in \Gamma_l \setminus\{(0,0)\}$, $\tau\in \R$, $\delta,\lambda>0$ be given. Define the number $\mathcal{W}_{l}(k,\lambda,\delta,\tau)$ by
    \begin{equation*}
        \mathcal{W}_{\Gamma_l}(k,\lambda,\delta,\tau) \coloneqq \# \big\{ n\in \Gamma_l \,\, | \, \, 2\pi|n| <\lambda, \, 2\pi|n-k| <\lambda, \, |2\pi|n| -2\pi|n-k| - \tau|<\delta   \big\}
    \end{equation*}
 Then we have the following
    \begin{equation*}
        \mathcal{W}_{\Gamma_l}(k,\lambda,\delta,\tau) \leq (2\pi^3 l_1l_2)^{-1}\delta\lambda^2 +\pi^{-1}((2\pi^2 l_1)^{-1}\delta +(\pi l_2)^{-1})\lambda +1.
    \end{equation*}
\end{lemma}

\begin{proof}
    Without loss of generality, we may assume that $k_1 \neq0 $. Furthermore, we will assume that $k_1>0$, for the case where $k_1<0$ the difference in the proofs is a change of signs and two inequalities flipping. Firstly, we notice that for $n\in \mathcal{W}_{\Gamma_l}(k,\lambda,\delta,\tau) $ we have
    \begin{equation*}
    \begin{aligned}
       \Big| 2\langle n,k \rangle - | k |^2 - \frac{\tau}{2\pi} (|n| + |n-k|) \Big|
       &=\Big| |n|^2 - |n-k|^2 - \frac{\tau}{2\pi} (|n| + |n-k|)\Big|
        \\
        &=\Big|(|n| - |n - k| - \frac{\tau}{2\pi})(|n| +|n-k|) \Big|
        \\
        &\leq \frac{\delta \lambda}{2\pi^2},
    \end{aligned}
    \end{equation*}
    where we have used the assumption that $n\in \mathcal{W}_{l}(k,\lambda,\delta,\tau) $. This implies that 
    \begin{equation*}
    \begin{aligned}
       -\frac{\delta \lambda}{4\pi^2}+\frac{  | k |^2 + \tau (|n| + |n-k|)}{2} \leq \langle n,k \rangle \leq \frac{\delta \lambda}{4\pi^2}+ \frac{ | k |^2 + \tau (|n| + |n-k|)}{2}.
    \end{aligned}
    \end{equation*}
    Writing out the inner product and using the assumption $k_1>0$, we get
    \begin{equation*}
    \begin{aligned}
      -\frac{\delta \lambda}{4\pi^2}+ \frac{| k |^2 + \tau (|n| + |n-k|)}{2 k_1} - \frac{n_2k_2}{k_1} \leq  n_1 \leq \frac{\delta \lambda}{4\pi^2}+\frac{ | k |^2 + \tau (|n| + |n-k|)}{2 k_1} - \frac{n_2k_2}{k_1}. 
    \end{aligned}
    \end{equation*}
    This implies that for each fixed $n_2$, there can be at most $(2\pi^2 l_1)^{-1} \delta \lambda +1$ possible values for $n_1$, where we have used $k_1\geq 1$. Since we can have at most $  (\pi l_2)^{-1} \lambda +1$ different values, it follows that
    \begin{equation*}
        \#\mathcal{W}_{\Gamma_l}(k,\lambda,\delta,\tau) \leq  (2\pi^3 l_1l_2)^{-1}\delta\lambda^2 +\pi^{-1}((2\pi^2 l_1)^{-1}\delta +(\pi l_2)^{-1})\lambda +1
    \end{equation*}
    This concludes the proof.
\end{proof}

\begin{proof}[Proof of Theorem~\ref{Thm:Main_theorem_tori}]
Firstly, we notice that by Lemma~\ref{LE:exstension_to_other_basis} it will suffice to prove the theorem for a particular basis. Hence, we will choose to work with the basis $\{\psi_{n} \}_{n\in \Gamma_l^{*}}$, where 
\begin{equation*}
\psi_{n}(x) = \frac{1}{ \sqrt{l_1 l_2}}e^{-i2\pi \langle n,x\rangle } \qquad\text{for all $n\in \Gamma_l^{*}$.} 
\end{equation*}
For this basis, we notice that for all $n,m\in\Gamma_l^{*}$ we have that
\begin{equation}\label{EQ:Main_theorem_tori_proof_1}
        \langle a \psi_n, \psi_m \rangle =  \frac{1}{ \sqrt{l_1 l_2}}  \int_{\mathbb{T}_{l}} a(x) \psi_{n-m}(x) \, dx = \frac{\hat{a}(n-m)}{ \sqrt{l_1 l_2}},
\end{equation}
where $\hat{a}(n-m)$ is the Fourier coefficient of $a$ at $n-m$.

For point $(1)$ we notice that for any $a\in L^2(\mathbb{T}_{l})$ we have from \eqref{EQ:Main_theorem_tori_proof_1} that every term in the sum is the integral over $a$ divided by the area of the torus. Hence the conclusion follows.

As point $(2)$ is a special case of point $(3)$ we will only prove the latter one. Let $\varepsilon>0$ and $\tau\in\R$ be given and consider $a \in L^2(\mathbb{T}_{l})$. Using the observation in \eqref{EQ:Main_theorem_tori_proof_1}, we get for all $\delta>0$ that 
\begin{equation}\label{EQ:torus_mix_1}
        \begin{aligned}
         \sum_{\substack{n\neq m: \lambda_n,\lambda_m\leq \lambda  \\ |\lambda_{n}-\lambda_{m}-\tau| < \delta(\varepsilon)}} \big| \langle a \psi_n , \psi_m\rangle  \big|^2  
         &= \frac{1}{ l_1 l_2}  \sum_{\substack{n\neq m: \lambda_n,\lambda_m\leq \lambda  \\ |\lambda_{n}-\lambda_{m}-\tau| < \delta(\varepsilon)}} |\hat{a}(n-m)|^2 
         \\
         &= \frac{1}{ l_1 l_2} \sum_{k\in\Gamma_l^{*}\setminus\{(0,0)\}} |\hat{a}(k)|^2\mathcal{W}_{\Gamma_l^{*}}(k,\lambda,\delta,\tau) ,
        \end{aligned}
    \end{equation}
    where the set $\mathcal{W}_{\Gamma_l^{*}}(k,\lambda,\delta,\tau)$ is defined as in Lemma~\ref{LE:simple_counting_lemma}. Applying this lemma and Plancherel, we obtain from \eqref{EQ:torus_mix_1} the inequality
    \begin{equation*}
        \begin{aligned}
         \sum_{\substack{n\neq m: \lambda_n,\lambda_m\leq \lambda  \\ |\lambda_{n}-\lambda_{m}-\tau| < \delta(\varepsilon)}} \big| \langle a \psi_n , \psi_m\rangle  \big|^2  
         \leq \frac{(2\pi^3 l_1l_2)^{-1}\delta\lambda^2 +\pi^{-1}((2\pi^2 l_1)^{-1}\delta +(\pi l_2)^{-1})\lambda +1}{ l_1 l_2} \norm{a}_{L^2(\mathbb{T}_{l})}^2.
        \end{aligned}
    \end{equation*}
    By Weyl's law it follows from this inequality that 
    \begin{equation*}
        \limsup_{\lambda\rightarrow \infty}\frac{1}{\mathcal{N}(\sqrt{-\Delta_{\mathbb{T}_{l}}},\lambda)} \sum_{\substack{n\neq m: \lambda_n,\lambda_m\leq \lambda  \\ |\lambda_{n}-\lambda_{m}-\tau| < \delta(\varepsilon)}} \big| \langle a \psi_n , \psi_m\rangle  \big|^2 \leq \frac{2\delta }{\pi^2 l_1 l_2} \norm{a}_{L^2(\mathbb{T}_{l})}^2.
    \end{equation*}
    From this estimate, we see that choosing $\delta$ sufficiently small, we obtain the stated result.
\end{proof}

\end{document}